\documentclass[11pt,a4paper]{amsart}

\usepackage{hyperref}
\usepackage{amsthm,latexsym,amssymb,amsmath,stmaryrd}
\usepackage{times,mathptmx}

\newcommand{\Ric}{\mathrm{Ric}}
\newcommand{\tp}{\tilde\times_\rho}
\newcommand{\tpt}{\tilde\times_{\tilde\rho}}

\newcommand{\Z}{\mathbb{Z}}
\newcommand{\R}{\mathbb{R}}
\newcommand{\C}{\mathbb{C}}
\renewcommand{\P}{\mathbb{P}}
\newcommand{\K}{\mathcal{K}}
\newcommand{\F}{\mathcal{F}}
\newcommand{\<}{\langle}
\renewcommand{\>}{\rangle}
\renewcommand{\phi}{\varphi}
\newcommand{\na}{\nabla}

\newcommand{\scal}{\mathrm{scal}}
\newcommand{\vol}{\mathrm{vol}}
\newcommand{\Iso}{\mathrm{Iso}^+}

\newcommand{\id}{\mathrm{id}}
\newcommand{\dm}{d^\nabla_-}
\newcommand{\Dm}{\mathcal{D}^-}

\newcommand{\sameauthor}{\rule{15mm}{0.5pt}: }

\newtheorem{thm}{Theorem}

\newtheorem*{thmnn}{Theorem}
\newtheorem*{cor}{Corollary}

\newtheorem{lem}{Lemma}

\theoremstyle{definition}

\newtheorem*{rem}{Remark}

\title[Geometrically formal $4$-manifolds with nonnegative curvature]{Geometrically formal $4$-manifolds with nonnegative sectional curvature}
\author{Christian B\"ar}
\address{Universit\"at Potsdam, Institut f\"ur Mathematik, Am Neuen Palais 10, 14469 Potsdam, Germany}
\email{baer@math.uni-potsdam.de}
\urladdr{http://geometrie.math.uni-potsdam.de/}


\begin{document}

\begin{abstract}
A Riemannian manifold is called geometrically formal if the wedge product of any two harmonic forms is again harmonic.
We classify geometrically formal compact $4$-manifolds with nonnegative sectional curvature.
If the sectional curvature is strictly positive, the manifold must be homeomorphic to $S^4$ or diffeomorphic to $\C\P^2$.

This conclusion stills holds true if the sectional curvature is strictly positive and we relax the condition of geometric formality to the requirement that the length of harmonic $2$-forms is not too nonconstant.
In particular, the Hopf conjecture on $S^2\times S^2$ holds in this class of manifolds.
\end{abstract}

\maketitle

\section{Introduction}

Compact Riemannian manifolds with positive sectional curvature are still poorly understood in the sense that, on the one hand, one knows relatively few examples and, on the other hand, most known topological obstructions against existence of a metric of positive curvature are in fact obstructions against weaker curvature conditions.
The only known connected orientable compact positively curved $4$-manifolds are $S^4$ and $\C\P^2$ and it has been conjectured that no further example exist \cite[p.~4]{Z}.
At the moment, a proof of this conjecture seems out of reach.
It would imply in particular that $S^2\times S^2$ does not carry a positively curved metric, a conjecture attributed to H.~Hopf.
Even the Hopf conjecture turned out to be notoriously difficult and is open to date.
Most currently available classification results for positively or nonnegatively curved manifolds require additional assumptions like a sufficiently high degree of symmetry. 
We refer to the beautiful surveys \cite{Wi,Z} for more on this.
In the present article we will consider nonnegatively curved $4$-manifolds which are geometrically formal.
In Theorem~\ref{thm:norm} the condition of geometric formality will be relaxed.
In Theorem~\ref{thm:alt} we consider positively curved $4$-manifolds which are not too nonsymmetric in the sense that the covariant derivative of their curvature tensor satisfies a suitable estimate.

A Riemannian manifold $M$ is called \emph{geometrically formal} if the wedge product of any two harmonic forms $M$ is again harmonic.
One motivation for studying such manifolds comes from the fact that geometrically formal manifolds are formal in the sense of Sullivan \cite[p.~43]{S}.
For a nice introduction to geometrically formal manifolds see \cite{K1}.

Let $M$ be a connected oriented geometrically formal $n$-manifold.
Given harmonic $k$-forms $\alpha$ and $\beta$ on $M$, the $n$-form $\alpha\wedge *\beta=\<\alpha,\beta\>\vol$ is again harmonic.
Hence the pointwise scalar product $\<\alpha,\beta\>$ is constant.
In particular, harmonic forms have constant length and cannot have zeros unless they are identically zero.

From this one easily deduces that the (real) Betti numbers are bounded by \mbox{$b_k(M) \le b_k(T^n)$} and, if $n$ is divisible by $4$, $b^\pm_{n/2}(M) \le b^\pm_{n/2}(T^n)$.
Here $T^n=\R^n/\Z^n$ denotes the $n$-torus.
Moreover, $b_1(M)\neq n-1$ so that $b_1(M) \in\{0,1,2,\ldots,n-2,n\}$, see \cite{K1} for details.

The purpose of this paper is to classify geometrically formal $4$-manifolds with nonnegative sectional curvature. 
In $3$ dimensions one has a good understanding of nonnegatively curved manifolds.
R.~Hamilton proved that any connected compact oriented $3$-dimensional Riemannian manifold with $\Ric\ge 0$ is diffeomorphic to a quotient of $S^3$ or $S^2\times\R$ or $\R^3$ by a group of fixed point free isometries in the standard metrics (\cite[Thm.~1.2]{Ham86}).
We will need a more precise formulation of this result:

\begin{thmnn}[Hamilton]\label{thm:3D}
Let $M$ be a connected compact oriented $3$-dimensional Riemannian manifold.
Suppose that the Ricci curvature of $M$ satisfies $\Ric\ge 0$.
Then one of the following holds:
\begin{enumerate}
\item\label{c:3D:sphere}
$M$ is diffeomorphic to a spherical spaceform or to $\R\P^3\sharp\R\P^3$;
\item\label{c:3D:twisted}
$M$ is isometric to a twisted product $S^2\tp S^1$ where $S^2$ carries a metric of nonnegative curvature;
\item \label{c:3D:flat}
$M$ is flat.
\end{enumerate}
\end{thmnn}

By a spherical spaceform we mean a manifold of the form $M=\Gamma\backslash S^n$ where \mbox{$\Gamma\subset\mathrm{O}(n+1)$} is a finite subgroup acting freely on $S^n$.
Spherical spaceforms and flat manifolds in $3$ dimensions are classified, see \cite[Thm.~3.5.5 and p.~224]{W} for the lists.
Twisted products will be explained in the next section.
The difference to Hamilton's formulation is that in Cases \eqref{c:3D:twisted} and \eqref{c:3D:flat} we not only have information about the diffeomorphism type but also about the metric.
In Section~\ref{sec:3D} we will show how to derive this version of Hamilton's theorem from the one given in \cite{Ham86}.
We give the proof because it introduces methods which will be crucial for treating the $4$-dimensional case.

In $4$ dimensions we show:

\begin{thm}\label{thm:4D}
Let $M$ be a connected compact oriented $4$-dimensional Riemannian manifold.
Suppose that $M$ is geometrically formal and the sectional curvature satisfies $K\ge 0$.
Then one of the following holds:
\begin{enumerate}
\item \label{c:4D:sphere}
$M$ is a rational homology $4$-sphere with finite fundamental group;
\item \label{c:4D:CP2}
$M$ is diffeomorphic to $\C\P^2$;
\item \label{c:4D:flat}
$M$ is flat;
\item \label{c:4D:twisted22}
$M$ is isometric to a twisted product $S^2\tp T^2$ where $T^2$ carries a flat metric and $S^2$ carries a metric of nonnegative curvature;
\item \label{c:4D:twisted31}
$M$ is isometric to a twisted product $\Sigma^3\tp S^1$, where $\Sigma^3$ is isometric to a spherical spaceform or to $\R\P^3\sharp\R\P^3$ with a metric satisfying $K\ge0$;
\item \label{c:4D:prod}
$M$ is isometric to $S^2\times S^2$ with product metric where both factors carry metrics with nonnegative curvature.
\end{enumerate}
\end{thm}

All cases in Theorem~\ref{thm:4D} do actually occur.
More precisely, the manifolds in Cases~\eqref{c:4D:flat}--\eqref{c:4D:prod} are geometrically formal and satisfy $K\ge0$.
Since flat manifolds and spherical spaceforms are classified, these cases are completely understood.

The standard sphere $S^4$ and $\C\P^2$ with the Fubini-Study metric are also geometrically formal and have strictly positive sectional curvature.
In fact, every metric on $S^4$ is geometrically formal;
hence we can replace the standard metric by any metric with $K\ge0$.
Thus Cases~\eqref{c:4D:sphere} and \eqref{c:4D:CP2} also occur.
In a previous version of the present article, the statements in Cases~\eqref{c:4D:sphere} and \eqref{c:4D:CP2} were weaker.
D.~Kotschick pointed out to the author that a corollary to the Cheeger-Gromoll splitting theorem implies finiteness of the fundamental group in Case~\eqref{c:4D:sphere}.
In Case~\eqref{c:4D:CP2}, it was originally only concluded that $M$ has the homology type of $\C\P^2$.
Here D.~Kotschick pointed out that the classification of symplectic $4$-manifolds with a metric of positive scalar curvature improves the result to its present form, compare also the proof of Theorem~8 in \cite{K3}.

Since the manifolds in Cases \eqref{c:4D:flat}, \eqref{c:4D:twisted22}, and \eqref{c:4D:twisted31} in Theorem~\ref{thm:4D} are not simply connected, we obtain:

\begin{cor}
Let $M$ be a simply connected compact oriented $4$-dimensional Riemannian manifold.
Suppose that $M$ is geometrically formal and the sectional curvature satisfies $K\ge0$.
Then one of the following holds:
\begin{enumerate}
\item \label{c:4D1:sphere}
$M$ is homeomorphic to $S^4$;
\item \label{c:4D1:CP2}
$M$ is diffeomorphic to $\C\P^2$;
\item \label{c:4D1:prod}
$M$ is isometric to $S^2\times S^2$ with product metric where both factors carry metrics with nonnegative curvature.
\end{enumerate}
\end{cor}

Note that the statement in Case~\eqref{c:4D1:sphere} is stronger than the corresponding one in Theorem~\ref{thm:4D}; 
isomorphic rational homology has been improved to homeomorphism.
A justification for this will be given in the proof of Theorem~\ref{thm:4D}, see Section~\ref{ssec:00}.
One might conjecture that even in Theorem~\ref{thm:4D} the conclusion in the first case should be that $M$ is diffeomorphic to $S^4$.

One may compare the corollary to an older result by B.~Kleiner where instead of geometric formality one assumes that there are sufficiently many isometries.
In this case there are two further possible homeomorphism types.
He shows (\cite[Thm.~1.0.2]{Kl} and \cite[Thm.~1]{SY}) that a simply connected compact oriented $4$-dimensional Riemannian manifold with $K\ge0$ which has a nontrivial isometric $\mathrm{U}(1)$-action must be homeomorphic to $S^4$, to $\C\P^2$, to $S^2\times S^2$, to $\C\P^2\sharp\C\P^2$ or to $\C\P^2\sharp\overline{\C\P}^2$.
Here $\overline{\C\P}^2$ denote the complex projective plane with the reversed orientation (the one not induced by the complex structure).
Using work of Fintushel \cite{Fi} one can replace ``homeomorphic'' by ``diffeomorphic'' in this theorem, compare e.g.\ \cite{GG} and the references given therein.

If the curvature is strictly positive, we can weaken the assumption of geometric formality quite a bit.
Recall that on a geometrically formal manifold all harmonic forms have constant length.
It suffices to assume that the length of harmonic $2$-forms is not ``too nonconstant''.
More precisely, we get

\begin{thm}\label{thm:norm}
Let $M$ be a connected compact oriented $4$-dimensional Riemannian manifold.
Suppose that the sectional curvature satisfies $K\ge \kappa >0$ and all harmonic $2$-forms $\omega$ satisfy $|d|\omega||\leq \sqrt{8\kappa}\cdot |\omega|$ wherever $\omega$ does not vanish.

Then $M$ is homeomorphic to $S^4$ or diffeomorphic to $\C\P^2$.
\end{thm}

In particular, $S^2\times S^2$ is ruled out and the Hopf conjecture holds in this class of manifolds.
We recover a result by W.~Seaman \cite{S1} stating that $M$ must be homeomorphic to $S^4$ or to $\C\P^2$ if the harmonic $2$-forms have constant length.

Also in the positively curved case there is an analogous result if one demands the existence of sufficiently many isometries.
Hsiang and Kleiner \cite{HK} showed that a connected compact oriented $4$-dimensional Riemannian manifold with $K>0$ which has a nontrivial isometric $\mathrm{U}(1)$-action must be homeomorphic to $S^4$ or to $\C\P^2$.
Again, Fintushel's results can be used to improve the conclusion from ``homeomorphic'' to ``diffeomorphic''.

Theorems~\ref{thm:4D} and \ref{thm:norm} will be proved in Sections~\ref{sec:4D} and \ref{sec:norm}, respectively.
The proof of Theorem~\ref{thm:norm} shows that the assumptions can be weakened as follows:
The estimate on $d|\omega|$ need not be demanded for \emph{all} harmonic $2$-forms but only for \emph{at least one} nontrivial selfdual harmonic \mbox{$2$-form} if \mbox{$b_2^+(M)>0$} and for at least one nontrivial antiselfdual harmonic $2$-form if \mbox{$b_2^-(M)>0$}.

A prominent class of geometrically formal manifolds is given by symmetric spaces.
They have parallel curvature tensor, $\nabla R=0$.
If we demand that the derivative of the curvature tensor is not too large, then again $S^2\times S^2$ is not possible.
In other words, a counterexample to the Hopf conjecture must be very nonsymmetric.
In fact, a bound on the differential of the scalar curvature and on the selfdual or antiselfdual Weyl curvature suffice.

To formulate the result we specify the norm on the relevant tensor spaces:
For any $(0,k)$-tensor field (e.g., any $k$-form) $B$ we consider
\[
|B|_{(k)} = \sup \frac{|B(X_1,\ldots,X_k)|}{|X_1|\cdots|X_k|}
\]
where the supremum is taken over all nonzero vectors $X_1,\ldots,X_k$.
Now we consider the Riemann curvature tensor $R$ and the Weyl curvature $W$ as $(0,4)$-tensors and their covariant derivatives $\nabla R$ and $\nabla W$ as $(0,5)$-tensors.
The selfdual part of $W$ is denoted by $W^+$ and its antiselfdual part by $W^-$.
Now the result is

\begin{thm}\label{thm:alt}
Let $M$ be a connected compact oriented $4$-dimensional Riemannian manifold with sectional curvature $K\ge 1$.
Suppose that $$|\na W^+|_{(5)} + \frac{1}{12}|d\scal|_{(1)}\le \frac{4}{\pi}.$$

Then $M$ is homeomorphic to $S^4$ or to $\underbrace{\C\P^2\sharp\cdots\sharp\C\P^2}_k$ with $1\le k\le 10^{238}$.
\end{thm}

Of course, the upper bound on $k$ in this theorem is far from optimal.
One would rather expect that $M$ has to be homeomorphic to $S^4$ or to $\C\P^2$.
By reversing the orientation of $M$, one can replace $W^+$ by $W^-$ in the assumption.
Note that there is no claim in the statement that the homeomorphism between $M$ and $\C\P^2\sharp\cdots\sharp\C\P^2$ has to be orientation preserving.
Theorem~\ref{thm:alt} may be compared to a result by M.~H.~Noronha \cite[Cor.~1]{N} where it is assumed that $\scal$ and the pointwise norm of $W^+$ are constant.

The main step for the proof of Theorem~\ref{thm:alt} was done in \cite{B}.
In Section~\ref{sec:alt} we will show how to derive the formulation given in Theorem~\ref{thm:alt} from the work in \cite{B}.

In \cite{K3} D.~Kotschick started to investigate low-dimensional geometrically formal manifolds which admit a (possibly different) metric of nonnegative scalar curvature and he obtained analogous classification results.

\emph{Acknowledgments.}
It is my pleasure to thank B.~Hanke, B.~Wilking, and W.~Ziller for helpful discussions.
I am grateful to D.~Kotschick for pointing out an error in the first version of this paper and for many very useful hints and references.
Moreover, I thank the \emph{Sonderforschungsbereich 647} funded by \emph{Deutsche Forschungsgemeinschaft} for financial support.

\section{Preliminaries}

\subsection{Abel-Jacobi map}
Let $M$ be a connected compact Riemannian manifold.
We assume that all harmonic $1$-forms on  $M$ have constant length.
This holds, for instance, if $M$ is geometrically formal or if $M$ has nonnegative Ricci curvature.
By polarization, all pointwise scalar products of harmonic $1$-forms are also constant.
Let $b=b_1(M)$.

Choose a basis $\{\theta_1,\ldots,\theta_b\}$ of the space of harmonic $1$-forms by forms with \emph{integral periods}.
Then there exist functions $F_j:M\to \R/\Z$ such that $dF_j=\theta_j$.
By \cite[Prop.~6.3]{BK},
the map $F=(F_1,\ldots,F_b):M\to T^b$ is a Riemannian submersion with minimal fibers.
Here $T^b=\R^b/\Z^b$ carries a flat Riemannian metric determined by the (constant) scalar products $\<\theta_i,\theta_j\>$.
The submersion $F$ is known as the \emph{Abel-Jacobi map}.
If $M$ is compact, Ehresmann's fibration theorem implies that the Abel-Jacobi map $F:M\to T^b$ is a fiber bundle map.
Moreover, $F^*:H^1(T^b,\R) \to H^1(M,\R)$ is an isomorphism.

\subsection{Twisted products}
Let $\Sigma$ be an oriented Riemannian manifold and denote by $\Iso(\Sigma)$ the group of orientation-preserving isometries of $\Sigma$.
Let $V$ be a \mbox{$b$-dimensional} Euclidean vector space and let $\Gamma\subset V$ be a lattice.
Let \mbox{$\rho:\Gamma\to\Iso(\Sigma)$} be a homomorphism.
Then $\Gamma$ acts isometrically on $\Sigma\times V$ by $\rho$ on the first factor and by translations on the second.
We denote the quotient $(\Sigma\times V)/\Gamma$ by $\Sigma\tp T^b$ and call it a \emph{twisted product}.
If $b=1$, then $\Sigma\tp \R/\Z$ is also known as the mapping torus of the map $\rho(1)$.
The projection onto the second factor $\Sigma\times V\to V$ induces a Riemannian submersion $\Sigma\tp T^b\to T^b$ with totally geodesic fibers isometric to $\Sigma$.
The torus $T^b=V/\Gamma$ carries the induced flat metric.
In case $\rho$ is trivial, the twisted product is just the usual Riemannian product, $\Sigma\tp T^b=\Sigma\times T^b$.

The following folklore lemma will be needed as a technical tool.

\begin{lem}\label{lem:twisted}
Let $M$ be a connected compact oriented Riemannian manifold.
Suppose that all harmonic $1$-forms on $M$ are parallel.

Then $M$ is isometric to a twisted product $\Sigma\tp T^b$ where $\Sigma$ is a connected compact oriented Riemannian manifold and $b=b_1(M)$.
\end{lem}

\begin{proof}
Let $\{\theta_1,\ldots,\theta_b\}$ be a basis of the space of harmonic $1$-forms with integral periods.
Let $F:M\to \R^b/\Z^b$ be the induced Abel-Jacobi map.
Denote the fiber of $F$ over $[0]\in \R^b/\Z^b$ by $\widehat\Sigma$.

Denote the metrically dual vector field to $\theta_j$ by $v_j$ and let $\Phi_j:\R\to\mathrm{Diff}(M)$ be its flow.
Since $v_j$ is parallel, it is a Killing vector field, hence $\Phi_j$ acts by orientation preserving isometries, $\Phi_j:\R\to\Iso(M)$.
Define $\Phi:\R^b\to\Iso(M)$ by $\Phi(t_1,\ldots,t_b):=\Phi_1(t_1)\circ\ldots\circ\Phi_b(t_b)$.
Since the vector fields $v_j$ are parallel they commute.
Hence the flows commute so that the order of the flow maps in the definition of $\Phi$ is irrelevant.
In particular, $\Phi$ is a group homomorphism.

The vector fields $v_j$ descend to parallel vector fields $F_*v_j$ on $\R^b/\Z^b$ and hence induce an action of $\R^b$ on $\R^b/\Z^b$ by translations.
Let $\widehat\Gamma\subset\R^b$ be the kernel of this action.
This $\R^b$-action is compatible with the $\R^b$-action on $M$ given by $\Phi$.
Thus $\Phi$ maps fibers of the Abel-Jacobi map to fibers.
In particular, we obtain a homomorphism $\widehat\rho:=\Phi|_{\widehat\Gamma}:\widehat\Gamma\to\Iso(\widehat\Sigma)$.

Now $\widehat\Sigma$ is compact and oriented but need not be connected.
Since $M$ is connected the action of $\widehat\Gamma$ on the set of connected components of $\widehat\Sigma$ is transitive.
In particular, all connected components of $\widehat\Sigma$ are isometric.
Denote one of the connected components of $\widehat\Sigma$ by $\Sigma$.
Let $\Gamma\subset\widehat\Gamma$ be the subgroup of elements mapping $\Sigma$ to $\Sigma$ under $\Phi$.
Then $\Gamma$ is a subgroup of $\widehat\Gamma$ of finite index and hence again a lattice in $\R^b$.
Put $\rho:=\Phi|_{\Gamma}:\Gamma\to\Iso(\Sigma)$.
Equip $\R^b$ with the metric which has the constant coefficients $\<\partial/\partial t_i,\partial/\partial t_j\> = \< v_i,v_j\>$.
Then the map $\Sigma\times\R^b \to M$, $(\sigma,t)\mapsto \Phi(t)(\sigma)$, is a local isometry and induces an isometry $\Sigma\tp (\R^b/\Gamma)\to M$.
\end{proof}

\subsection{Examples of geometrically formal manifolds}
Now we briefly discuss examples of geometrically formal manifolds.
The more cohomology a manifold carries, the more restrictive is the assumption of geometric formality.
If $M$ is diffeomorphic to $S^n$ (or, more generally, to a rational homology sphere), then $M$ is geometrically formal with any Riemannian metric.
If $M$ is diffeomorphic to $T^n$, then $M$ is geometrically formal if and only if $M$ is flat \cite[Thm.~7]{K1}.
All Riemannian symmetric spaces are geometrically formal.
Further examples of homogeneous but nonsymmetric spaces which are geometrically formal can be found in \cite{KT}.
If $M$ is a closed oriented surface of genus $\ge2$, then $M$ does not admit a Riemannian metric making it geometrically formal because every $1$-form must have zeros and therefore cannot have constant length.

If $M_1$ and $M_2$ are geometrically formal, then so is the Riemannian product $M_1\times M_2$.
If $\Sigma$ is a rational homology sphere, then constant functions and constant multiples of the volume form are the only harmonic forms.
In particular, all harmonic forms on $\Sigma$ are parallel.
Similarly, all harmonic forms on a flat torus $T^b$ are parallel.
These forms induce parallel forms on $\Sigma\times\R^b$ which are invariant under the $\Z^b$-action induced by a homomorphism $\rho:\Z^b\to\Iso(\Sigma)$.
Thus they descend to parallel forms on the twisted product $M:=\Sigma\tp T^b$.
By the Leray-Hirsch theorem \cite[Thm.~5.11]{BT} applied to the fibering $\Sigma\hookrightarrow M\to T^b$ there are no further harmonic forms on $M$.
Hence all harmonic forms on the twisted product $M$ are parallel and so $M$ is geometrically formal.

\section{The 3-dimensional case}
\label{sec:3D}

To warm up we consider the $3$-dimensional case and prove the corollary to Hamilton's theorem.
So let $M$ be a connected compact oriented $3$-dimensional Riemannian manifold with nonnegative Ricci curvature, $\Ric\ge0$.
The Bochner formula for $1$-forms tells us that every harmonic $1$-form $\theta$ satisfies
\[
0=(\Delta\theta,\theta) = \|\na\theta\|^2 + (\Ric\theta,\theta) \ge  \|\na\theta\|^2 .
\]
Here $(\cdot,\cdot)$ denotes the $L^2$-scalar product and $\|\cdot\|$ the $L^2$-norm.
Hence every harmonic $1$-form is parallel.
In particular, $b_1(M) \leq 3$.
If there are two linearly independent harmonic (hence parallel) $1$-forms $\theta_1$ and $\theta_2$, then $*(\theta_1\wedge\theta_2)$ is also parallel, hence harmonic and we have three linearly independent harmonic $1$-forms.
Therefore, the first Betti number can take the values $b=b_1(M)\in\{0,1,3\}$ only.

If $b_1(M)=3$, then the Abel-Jacobi map yields a covering $M\to T^3$.
Thus $M$ is itself diffeomorphic to a torus.
By \cite[Cor.~A on p.~94]{GL}, $M$ must be flat and we are in Case~\eqref{c:3D:flat}.
Alternatively, one may argue that the tangent bundle of $M$ is trivialized by parallel vector fields.
Since they satisfy $(\Ric\theta,\theta)=0$ this shows $\Ric\equiv0$, hence $M$ is flat.

Let $b_1(M)=1$.
By Lemma~\ref{lem:twisted}, $M$ is isometric to a twisted product $\Sigma\tp S^1$ where $\Sigma$ is a connected compact oriented surface.
Since $\Sigma$ is a totally geodesic submanifold of $M$ it must have curvature $\ge0$.
If $\Sigma$ is a torus, then $\Sigma$ is flat.
Hence $M$ is flat and we are again in Case~\eqref{c:3D:flat}.
If $\Sigma$ is a sphere, then we are in Case~\eqref{c:3D:twisted}.

Let $b_1(M)=0$.
It is only in this case that we use Hamilton's theorem.
If $M$ is diffeomorphic to a flat manifold, then again by  \cite[Cor.~A]{GL}, $M$ must be flat and we are in Case~\eqref{c:3D:flat}.
If $M$ is diffeomorphic to a quotient of $S^2\times\R$, then $M$ must be diffeomorphic to $S^2\times S^1$ or to $\R\P^3\sharp\R\P^3$ \cite[p.~457]{Sc}.
Now $M$ cannot be diffeomorphic to $S^2\times S^1$ because $b_1(S^2\times S^1)=1$.
If $M$ is diffeomorphic to $\R\P^3\sharp\R\P^3$ or to a spherical spaceform, then we are in Case~\eqref{c:3D:sphere}.
This concludes the proof of the corollary.

\begin{rem}
Hamilton's theorem is based on Ricci flow and is a highly nontrivial result.
Without referring to Hamilton's theorem and Ricci flow, the above proof still yields a weaker result.
Namely, it shows that for any connected compact oriented $3$-dimensional Riemannian manifold $M$ with $\Ric\ge0$ one of the following holds:
\begin{enumerate}
\item
$M$ is a rational homology $3$-sphere;
\item
$M$ is isometric to a twisted product $S^2\tp S^1$ where $S^2$ carries a metric of nonnegative curvature;
\item
$M$ is flat.
\end{enumerate}
The point of Hamilton's theorem is that the rational homology spheres occurring in Case~\eqref{c:3D:sphere} must be spherical spaceforms or $\R\P^3\sharp \R\P^3$.
Of course, there are further rational homology $3$-spheres such as one of the six diffeomorphism types of flat $3$-manifolds.
\end{rem}

\section{Proof of Theorem~\ref{thm:4D}}
\label{sec:4D}

Let $M$ be a geometrically formal connected compact oriented $4$-dimensional Riemannian manifold with nonnegative sectional curvature, $K\ge0$.
The first Betti number can take the values $b=b_1(M)\in\{0,1,2,4\}$.
We consider the possibilities separately.

\subsection{Case \emph{b}\,=\,4}
Since $M$ satisfies $\Ric\ge0$, the $4$ linearly independent harmonic $1$-forms are parallel, hence $M$ is flat.
We are in Case~\eqref{c:4D:flat} of Theorem~\ref{thm:4D}.

\subsection{Case \emph{b}\,=\,2}
Arguing as in the proof of the corollary to Hamilton's theorem we conclude that $M$ is isometric to $\Sigma\tp T^2$ where $\Sigma$ is a connected compact oriented surface with curvature~$\ge0$.
If $\Sigma$ is a torus, we are in Case~\eqref{c:4D:flat} of Theorem~\ref{thm:4D}, if $\Sigma$ is a sphere we are in Case~\eqref{c:4D:twisted22} of Theorem~\ref{thm:4D}.

\subsection{Case \emph{b}\,=\,1}
Again as in the proof of the corollary we get that $M$ is isometric to $\Sigma\tp S^1$ where now $\Sigma$ is a connected compact oriented $3$-manifold with $K\ge0$.
By the corollary, $\Sigma$ can be flat or a spherical spaceform or $\R\P^3\sharp\R\P^3$ or be isometric to $S^2\tpt S^1$.
If $\Sigma$ is flat, then $M$ is flat and we are in Case~\eqref{c:4D:flat} of Theorem~\ref{thm:4D}.
If $\Sigma$ is a spherical spaceform or $\R\P^3\sharp\R\P^3$, then we are in Case~\eqref{c:4D:twisted31} of Theorem~\ref{thm:4D}.

We show that $\Sigma=S^2\tpt S^1$ is not possible.
Assume that $\Sigma$ is isometric to $S^2\tpt S^1$ where $S^2$ carries a metric with curvature $\ge0$.
Now $M$ fibers over $S^1$ with totally geodesic fibers $S^2\tpt S^1$ and the fibers fiber over $S^1$ with totally geodesic fibers $S^2$.
Thus $M$ carries a totally geodesic foliation with leaves diffeomorphic to $S^2$.
Since $M$ is locally isometric to $S^2\times \R\times\R$ we get a second $2$-dimensional totally geodesic foliation on $M$, perpendicular to the first one, with flat leaves.

Since $M$ has a harmonic $1$-form without zeros, the Euler number of $M$ vanishes, $\chi(M)=0$.
From $b_0(M)=b_1(M)=b_3(M)=b_4(M)=1$ we conclude $b_2(M)=0$.

On the other hand, the area $2$-form $\alpha$ of the first foliation (with leaves $S^2$) is parallel since both foliations are totally geodesic.
Hence $\alpha$ is harmonic and represents a nontrivial cohomology class in $H^2(M,\R)$.
This contradicts $b_2(M)=0$ showing that the subcase $\Sigma=S^2\tpt S^1$ cannot occur.

\subsection{Case \emph{b}\,=\,0}
\label{ssec:00}

A priori, $b_2^\pm(M)$ could take the values $0,1,2$, and $3$ but Kotschick has shown \cite[p.~527]{K1} that only $0$ and $1$ can occur for geometrically formal $4$-manifolds with $b_1(M)=0$.
We have to consider the various possibilities.

\emph{Subcase $b_2^+(M)=b_2^-(M)=0$}:
In this case $M$ is a rational homology sphere and we are in Case~\eqref{c:4D:sphere} of Theorem~\ref{thm:4D}.
Since the Euler number of $M$ is $\chi(M)=2\neq0$, \cite[Cor.~9.4]{CG} implies that $\pi_1(M)$ is finite.
If $M$ is simply connected, Freedman's theorem \cite{F} implies that $M$ is homeomorphic to $S^4$.
Thus we are in Case~\eqref{c:4D1:sphere} of the corollary to Theorem~\ref{thm:4D}.

\emph{Subcase $b_2^+(M)=1$ and $b_2^-(M)=0$}:
In this case $M$ is a rational homology $\C\P^2$.
Moreover, $M$ is a symplectic manifold, the symplectic form being given by the harmonic selfdual $2$-form (which has no zeros, by geometric formality).
Furthermore, $M$ cannot be flat because it has positive Euler number $3$.
Thus the scalar curvature is nonnegative and not identically zero, hence $M$ admits a metric with positive scalar curvature.
Now the classification of symplectic $4$-manifolds carrying a metric with positive scalar curvature (\cite[Thm.~C]{L} or \cite[Thm.~1.1]{OO}) implies that $M$ is diffeomorphic to a blow-up of $\C\P^2$ or of a ruled surface.
Among those manifolds only $\C\P^2$ itself is a rational homology $\C\P^2$.
Hence $M$ is diffeomorphic to $\C\P^2$.
(We learned this symplectic argument from the proof of Theorem~8 in \cite{K3}.)

\emph{Subcase $b_2^+(M)=0$ and $b_2^-(M)=1$}:
After reversing the orientation of $M$, this subcase reduces to the previous subcase.

\emph{Subcase $b_2^+(M)=b_2^-(M)=1$}:
First we observe that die Euler number is given by 
\begin{align*}
\chi(M)
&= b_0(M)-b_1(M)+b_2^+(M)+b_2^-(M)-b_3(M)+b_4(M) \\
&= 1-0+1+1-0+1=4.
\end{align*}
Now let $\omega^+$ be a harmonic selfdual $2$-form and $\omega^-$ a harmonic antiselfdual \mbox{$2$-form}.
Since harmonic forms have constant length, we may assume $|\omega^\pm|\equiv 1$.
Now \mbox{$\eta:=\omega^+ + \omega^-$} satisfies
\begin{align*}
\eta\wedge\eta 
&= \omega^+\wedge\omega^+ + \omega^-\wedge\omega^- 
= \left(\<\omega^+,*\omega^+\>+\<\omega^-,*\omega^-\>\right)\vol  \\
&= \left(|\omega^+|^2-|\omega^-|^2\right)\vol
=0 .
\end{align*}
This implies that $\eta$ is decomposable at each point of $M$.

Denote the curvature endomorphism in the Bochner formula for $2$-forms by $\K$, i.e.,
\[
\Delta = \na^*\na + \K
\]
on $\Omega^2(M)$.
For decomposable $2$-forms $\K$ has a nice expression.
If we write, at a given point $p\in M$, $\eta=\sqrt{2}\cdot e_1\wedge e_2$ where $e_1,e_2,e_3,e_4$ is a suitable orthonormal basis of $T_p^*M$, then formula (3) in \cite[p.~354]{S0} for $\K$ easily implies
\[
\<\eta,\K\eta\>
= 
2\left(K_{13} + K_{14} + K_{23} + K_{24}\right) .
\]
Here $K_{ij}$ denotes the sectional curvature of the plane spanned by $e_i$ and $e_j$.
Hence
\[
0=(\Delta\eta,\eta)\ge\|\na\eta\|^2.
\]
Thus $\eta$ is parallel.
The same reasoning shows that \mbox{$\eta':=\omega^+ - \omega^-$} is parallel, hence we have two perpendicular parallel decomposable $2$-forms.
This corresponds to an orthogonal splitting $TM=E_1\oplus E_2$ of the tangent bundle into parallel oriented plane bundles.
Parallelity implies involutivity, hence the distributions $E_1$ and $E_2$ are integrable.
Parallelity of $E_1$ and $E_2$ also implies that the second fundamental form of the leaves vanishes, i.e., the corresponding foliations $\F_1$ and $\F_2$ of $M$ are totally geodesic.
Moreover, the universal covering $\widetilde M$ of $M$ is isometric to $F_1\times F_2$ equipped with the product metric where $F_i$ are simply connected complete surfaces with curvature $\ge0$.
Under the projection $\widetilde M\to M$, the factors of $\widetilde M=F_1\times F_2$ are mapped onto the leaves of the foliations $\F_1$ and $\F_2$, respectively.

If $F_1$ and $F_2$ are diffeomorphic to $S^2$, then either $M=F_1\times F_2$ and we are in Case~\eqref{c:4D:prod} of Theorem~\ref{thm:4D} or $M$ is a proper quotient of $S^2\times S^2$.
Since $S^2\times S^2$ has Euler number~$4$, the Euler number of $M$ would have to be $2$ or $1$ if $M$ were a proper quotient.
Since $\chi(M)=4$, this cannot be the case.

It remains to see what happens if $F_1$ or $F_2$ (or both) is diffeomorphic to $\R^2$.
We show that this cannot occur.
Namely, if one of the factors is diffeomorphic to $\R^2$, then $\widetilde M$ is diffeomorphic to $\R^4$ or to $S^2\times \R^2$.
In either case, $\pi_1(M)$ is infinite and the Cheeger-Gromoll splitting theorem \cite[Cor.~9.4]{CG} implies $\chi(M)=0$.
This contradicts $\chi(M)=4$ and concludes the proof of Theorem~\ref{thm:4D}.

\section{Proof of Theorem~\ref{thm:norm}}
\label{sec:norm}

We start by observing that the assumption $|d|\omega|| \leq \sqrt{8\kappa}\cdot |\omega|$ implies that $\omega$ vanishes nowhere (unless it is identically zero).
Namely, pick $p\in M$ with $\omega(p)\neq 0$ and assume that the zero locus of $\omega$ is nonempty.
Let $q$ be a closest point to $p$ where $\omega$ vanishes.
Join $p$ and $q$ by a shortest geodesic $c:[0,L]\to M$, parametrized by arc-length.
Here $c(0)=p$, $c(L)=q$, and $L$ is the Riemannian distance of $p$ and $q$.
Now consider the function $f:[0,L)\to \R^+$ given by $f(t)=1/|\omega(c(t))|$.
We compute
\[
f' = \frac{-\<\dot{c},d|\omega|\>}{|\omega|^2} \leq \frac{|\dot{c}|\cdot|d|\omega||}{|\omega|^2} \leq \frac{\sqrt{8\kappa}}{|\omega|} = \sqrt{8\kappa}\cdot f.
\]
The Gronwall lemma implies $f(t)\leq f(0)\cdot\exp(\sqrt{8\kappa}t)$, in other words, \mbox{$|\omega(c(t))| \ge |\omega(p)|\cdot\exp(-\sqrt{8\kappa}t)$}.
This contradicts $|\omega(c(t))|\to 0$ as $t\to L$.

The crucial point in the proof of Theorem~\ref{thm:norm} is to show that $b_2^+(M)$ and $b_2^-(M)$ cannot be both positive.
Let us assume the contrary so that we can find a nontrivial selfdual $2$-form $\omega^+$ and a nontrivial antiselfdual $2$-form $\omega^-$.

We consider the form bundles $\Lambda^kT^*M\otimes \Lambda^-T^*M$ twisted with the bundle of antiselfdual $2$-forms together with its natural connection induced by the Levi-Civita connection.
Let $\dm$ be the exterior differential on this twisted form bundle.
Let $\Dm:=\dm+(\dm)^*$ be the associated generalized Dirac operator acting on sections of $\Lambda^*T^*M\otimes \Lambda^-T^*M=\bigoplus_{k=0}^4\Lambda^kT^*M\otimes \Lambda^-T^*M$.
We apply $\Dm$ to $\omega^+\otimes\omega^-$ and we get, using that $\omega^+$ is harmonic,
\begin{align}
|\Dm(\omega^+\otimes\omega^-)|^2
&=
\big|(d+d^*)\omega^+ \otimes \omega^- + \sum_{k=1}^4 e_k\cdot\omega^+\otimes\nabla_k\omega^-\big|^2 \notag\\
&=
\sum_{k,\ell=1}^4 \<e_k\cdot\omega^+,e_\ell\cdot\omega^+\>\<\nabla_k\omega^-,\nabla_\ell\omega^-\> \notag\\
&=
|\omega^+|^2\cdot\sum_{k=1}^4 |\nabla_k\omega^-|^2 \notag\\
&=
|\omega^+|^2\cdot|\nabla\omega^-|^2 . \label{eq:D-o+o-}
\end{align}
Here $e_1,\ldots,e_4$ denotes a local orthonormal tangent frame and \mbox{$e_k\cdot\omega = e_k^*\wedge\omega - e_k\lrcorner\omega$} is the Clifford multiplication.
Since the local frame can be chosen such that, up to a multiple, $\omega^+ = e_1\wedge e_2 + e_3\wedge e_4$, we see that $e_k\cdot\omega^+$ and $e_\ell\cdot\omega^+$ are perpendicular unless $k=\ell$.

For the covariant derivative of $\omega^+\otimes\omega^-$ we obtain
{\allowdisplaybreaks
\begin{align}
|\nabla(\omega^+\otimes&\omega^-)|^2
=
\sum_{k=1}^4 |\nabla_k\omega^+\otimes\omega^-+\omega^+\otimes\nabla_k\omega^-|^2  \notag\\
&=
|\nabla\omega^+|^2\cdot |\omega^-|^2 + |\omega^+|^2\cdot |\nabla\omega^-|^2 
+ 2 \sum_{k=1}^4 \<\nabla_k\omega^+,\omega^+\>\<\omega^-,\nabla_k\omega^-\> \notag\\
&=
|\nabla\omega^+|^2\cdot |\omega^-|^2 + |\omega^+|^2\cdot |\nabla\omega^-|^2 
+ \frac12 \sum_{k=1}^4 \partial_k|\omega^+|^2\cdot\partial_k|\omega^-|^2 \notag\\
&=
|\nabla\omega^+|^2\cdot |\omega^-|^2 + |\omega^+|^2\cdot |\nabla\omega^-|^2 
+ \frac12 \<d|\omega^+|^2,d|\omega^-|^2\> \notag\\
&=
|\nabla\omega^+|^2\cdot |\omega^-|^2 + |\omega^+|^2\cdot |\nabla\omega^-|^2 
+ 2 \<d|\omega^+|,d|\omega^-|\> |\omega^+|\cdot|\omega^-|  \notag\\
&\ge
|\nabla\omega^+|^2\cdot |\omega^-|^2 + |\omega^+|^2\cdot |\nabla\omega^-|^2 
- 2 |d|\omega^+||\cdot|d|\omega^-||\cdot |\omega^+|\cdot|\omega^-| .
\label{eq:nabla-o+o-}
\end{align}
}%
By the refined Kato inequality for harmonic $2$-forms in $4$ dimensions \cite[Thm.~1]{S2} we have
\[
|d|\omega^\pm|| \leq \sqrt{\frac23} \cdot |\nabla\omega^\pm|.
\]
This yields the estimate
\begin{align}
|d|\omega^+||\cdot|d|\omega^-||\cdot |\omega^+|\cdot|\omega^-| 
&\le
\frac23 \cdot |\nabla\omega^+|\cdot |\nabla\omega^-|\cdot |\omega^+|\cdot|\omega^-|  \notag\\
&\le
\frac13 \left( |\nabla\omega^+|^2\cdot |\omega^-|^2 + |\omega^+|^2\cdot |\nabla\omega^-|^2  \right) .  \label{eq:Absch1}
\end{align}
By the assumption in Theorem~\ref{thm:norm} we also have the estimate
\begin{equation}
|d|\omega^+||\cdot|d|\omega^-||\cdot |\omega^+|\cdot|\omega^-| 
\le
8\kappa\cdot |\omega^+|^2\cdot|\omega^-|^2 . \label{eq:Absch2}
\end{equation}
We use estimate \eqref{eq:Absch1} for $\frac32 |d|\omega^+||\cdot|d|\omega^-||\cdot |\omega^+|\cdot|\omega^-|$ and \eqref{eq:Absch2} for \mbox{$\frac12 |d|\omega^+||\cdot|d|\omega^-||\cdot |\omega^+|\cdot|\omega^-|$} in \eqref{eq:nabla-o+o-} and we obtain
\begin{equation}
|\nabla(\omega^+\otimes\omega^-)|^2
\ge
\frac12\left(|\nabla\omega^+|^2\cdot |\omega^-|^2 + |\omega^+|^2\cdot |\nabla\omega^-|^2 \right)
- 4\kappa |\omega^+|^2\cdot|\omega^-|^2 .
\label{eq:Absch3}
\end{equation}
Let $\K$ be the curvature endomorphism in the Bochner formula for the Hodge-Laplacian on forms.
Then the Weitzenb\"ock formula for $\Dm$ says (see \cite[Thm.~8.17]{LM} and its proof)
\begin{equation}
(\Dm)^2 = \nabla^*\nabla + \K\otimes\id_{\Lambda^-} + \sum_{i<j} e_i\cdot e_j \otimes R^{\Lambda^-}(e_i,e_j) .
\label{eq:weitz}
\end{equation}
Here $e_i\cdot e_j$ acts by Clifford multiplication on the first tensor factor.
Now we observe for each summand in the last term applied to $\omega^+\otimes \omega^-$:
\begin{align*}
\<e_i\cdot e_j\cdot\omega^+\otimes R(e_i,e_j)\omega^-,\omega^+\otimes\omega^- \>
&=
\<e_i\cdot e_j\cdot\omega^+,\omega^+\>\<R(e_i,e_j)\omega^-,\omega^-\> \\
&=
-\<´e_j\cdot\omega^+,e_i\cdot \omega^+\>\<R(e_i,e_j)\omega^-,\omega^-\> \\
&=
0 .
\end{align*}
Thus \eqref{eq:weitz} yields
\begin{equation}
\|\Dm(\omega^+\otimes\omega^-)\|^2
=
\|\nabla(\omega^+\otimes\omega^-)\|^2 + \int_M\<\K\omega^+,\omega^+\>|\omega^-|^2 \, d\vol .
\label{eq:wo+o-}
\end{equation}
Inserting \eqref{eq:D-o+o-} and \eqref{eq:Absch3} into \eqref{eq:wo+o-} we find
\begin{align}
\frac12 \int_M|\omega^+|^2\cdot|\nabla\omega^-|^2 \, d\vol
\ge&
\frac12 \int_M |\nabla\omega^+|^2\cdot |\omega^-|^2 \, d\vol
- 4\kappa \int_M |\omega^+|^2\cdot|\omega^-|^2 \, d\vol \notag\\
& + \int_M\<\K\omega^+,\omega^+\>|\omega^-|^2 \, d\vol .
\label{eq:Absch4}
\end{align}
Interchanging the roles of $\omega^+$ and $\omega^-$ we obtain
\begin{align}
\frac12 \int_M|\nabla\omega^+|^2 \cdot|\omega^-|^2\, d\vol
\ge&
\frac12 \int_M|\omega^+|^2 \cdot  |\nabla\omega^-|^2\, d\vol
- 4\kappa \int_M |\omega^+|^2\cdot|\omega^-|^2 \, d\vol \notag\\
& + \int_M|\omega^+|^2\<\K\omega^-,\omega^-\> \, d\vol .
\label{eq:Absch5}
\end{align}
Adding \eqref{eq:Absch4} and \eqref{eq:Absch5} we get
\begin{equation}
8\kappa \int_M |\omega^+|^2\cdot|\omega^-|^2 \, d\vol
\ge
\int_M\left(\<\K\omega^+,\omega^+\>|\omega^-|^2 +|\omega^+|^2\<\K\omega^-,\omega^-\> \right)\, d\vol .
\label{eq:Absch6}
\end{equation}
As in the previous section we see that $\eta=|\omega^-|\cdot\omega^+ + |\omega^+|\cdot\omega^-$ is a decomposable $2$-form and we obtain again
\[
\<\eta,\K\eta\> \ge 4\kappa|\eta|^2 = 8\kappa |\omega^+|^2|\omega^-|^2 .
\]
Now $\<\eta,\K\eta\>$ is precisely the integrand on the RHS of \eqref{eq:Absch6} and we obtain the opposite inequality 
\[
\int_M\left(\<\K\omega^+,\omega^+\>|\omega^-|^2 +|\omega^+|^2\<\K\omega^-,\omega^-\> \right)\, d\vol\ge
8\kappa \int_M |\omega^+|^2\cdot|\omega^-|^2 \, d\vol .
\]
Thus we have equality in \eqref{eq:Absch6} and therefore we must have equality in all estimates which we used to derive \eqref{eq:Absch6}.
In particular, we have $|d|\omega^+||\equiv \sqrt{8\kappa}|\omega^+|$.
On the other hand, $|\omega^\pm|$ must achieve its maximum at some point, a contradiction.

We have shown $b_2^+(M)=0$ or $b_2^-(M)=0$.
In other words, the intersection form of $M$ is positive or negative definite.
If $b_2^+(M)=b_2^-(M)=0$, then $M$ is a simply-connected homology $4$-sphere, hence homeomorphic to $S^4$ \cite{F}.
Upon reversing the orientation if necessary, we may assume $b_2^+(M)\ge1$ and $b_2^-(M)=0$.
Any nontrivial selfdual harmonic $2$-form is a symplectic form because it has no zeros.
By the classification of symplectic manifolds carrying a metric with positive scalar curvature (\cite[Thm.~C]{L} or \cite[Thm.~1.1]{OO}) $M$ is diffeomorphic to a blow-up of $\C\P^2$ or of a ruled surface.
Among those complex surfaces, $\C\P^2$ is the only simply-connected one with definite intersection form.
Hence $M$ must be diffeomorphic to $\C\P^2$.
As mentioned in Section~\ref{sec:4D}, this symplectic argument was already used in the proof of Theorem~8 in \cite{K3}.
This concludes the proof of Theorem~\ref{thm:norm}.

\section{Proof of Theorem~\ref{thm:alt}}
\label{sec:alt}

We define the $(0,4)$-tensor field $T:=W-\frac{\scal}{12}\cdot g \owedge g$ where $W$ is the Weyl curvature, $g$ is the Riemannian metric, and $\owedge$ denotes the Kulkarni-Nomizu product, see \cite[Def.~1.110]{Bes2}.
For any selfdual $2$-form $\omega$ it was shown in the proof of Prop.~3 in \cite{B} that
\begin{align*}
\frac{|\<(\nabla_X\K)\omega,\omega\>|}{|\omega|^2}
=\,& |(\nabla_X T)(e_1,e_2,e_1,e_2)+2(\nabla_X T)(e_1,e_2,e_3,e_4)  \\
& +(\nabla_X T)(e_3,e_4,e_3,e_4)|
\end{align*}
where $X$ is any tangent vector and $e_1,\ldots,e_4$ a suitable orthonormal tangent basis.
If $X$ has unit length we conclude
\begin{align*}
\frac{|\<(\nabla_X\K)\omega,\omega\>|}{|\omega|^2}
\le\,& 
|(\nabla_X W^+)(e_1,e_2,e_1,e_2)+2(\nabla_X W^+)(e_1,e_2,e_3,e_4)  \\
& +(\nabla_X W^+)(e_3,e_4,e_3,e_4)| + \bigg|\frac{\partial_X\scal}{12}\bigg|\cdot
|(g\owedge g)(e_1,e_2,e_1,e_2) \\
&+2(g\owedge g)(e_1,e_2,e_3,e_4)  +(g\owedge g)(e_3,e_4,e_3,e_4)| \\
\le\,&
4|\nabla W^+|_{(5)} + \frac{1}{12}|d\scal|_{(1)}\cdot(2+2\cdot0 +2) \\
=\,&
4|\nabla W^+|_{(5)} + \frac{1}{3}|d\scal|_{(1)} \\
\le\,& \frac{16}{\pi}.
\end{align*}
It is shown in \cite{B} that this estimate together with $K\ge1$ implies that the intersection form of $M$ is definite.
By Donaldson's theorem \cite{D}, the intersection form must be diagonalizable over $\Z$.
Freedman's theorem implies that $M$ is homeomorphic to a $k$-fold connected sum of $\C\P^2$'s with $k\in\{0,1,2,\ldots\}$.
In particular, the total Betti number of $M$ is $\sum_{m=0}^4 b_m(M) = 2+k$.
By a result of Gromov \cite{G}, the total Betti number of a nonnegatively curved $n$-manifold is bounded by a constant $C(n)$ only depending on the dimension of $M$.
Abresch \cite[p.~477]{A} showed that the constant $C(n)=\exp(6n^3+9n^2+4n+4)$ does the job.
In $4$ dimensions this yields $\sum_{m=0}^4 b_m(M) \le 10^{238}$ and concludes the proof of Theorem~\ref{thm:alt}.

\end{document}